\documentclass[11pt,twoside]{amsart}
\usepackage[all]{xy}
        \usepackage {amssymb,latexsym,amsthm,amsmath,mathtools,dsfont}
        \usepackage{enumitem,color}
        \usepackage[mathscr]{euscript}

\usepackage{mathtools}

\long\def\symbolfootnote[#1]#2{\begingroup%
\def\thefootnote{\fnsymbol{footnote}}\footnote[#1]{#2}\endgroup}
\newcommand{\D}{\ensuremath{\mathcal{D}}}
\newcommand{\C}{\ensuremath{\mathcal{C}}}
\newcommand{\Z}{\ensuremath{\mathbb{Z}}}
\newcommand{\g}{\textup{Gal}}
\newcommand{\h}{\textup{Hol}}
\newcommand{\p}{\textup{Perm}}
\newcommand{\au}{\textup{Aut}}
\newcommand{\cy}{\mathfrak{C}}
\newcommand{\di}{\mathfrak{D}}
\newcommand{\mi}{\mathfrak{M}}

\newcommand{\ra}{\mathfrak{R}}

\makeatletter
\def\imod#1{\allowbreak\mkern10mu({\operator@font mod}\,\,#1)}
\makeatother
\makeatletter
\renewcommand*\env@matrix[1][*\c@MaxMatrixCols c]{%
  \hskip -\arraycolsep
  \let\@ifnextchar\new@ifnextchar
  \array{#1}}
\makeatother

\newtheorem{theorem}{Theorem}[section]
\newtheorem{lemma}[theorem]{Lemma}
\newtheorem{corollary}[theorem]{Corollary}
\newtheorem{proposition}[theorem]{Proposition}
\newtheorem*{theorem*}{Theorem}
\theoremstyle{definition}
\newtheorem{definition}[theorem]{Definition}
\newtheorem{remark}[theorem]{Remark}

\numberwithin{equation}{section}

\newcommand{\ignore}[1]{}

\newcommand{\mynote}[1]{}
\begin{document}
\setcounter{section}{0}
\title{On $\mathbb{Z}_N\rtimes\mathbb{Z}_2$-Hopf-Galois structures}
\author{Namrata Arvind, Saikat Panja}
\email{namchey@gmail.com, panjasaikat300@gmail.com}
\address{IISER Pune, Dr. Homi Bhabha Road, Pashan, Pune 411 008, India}
\thanks{The first named author is partially supported by the IISER Pune research fellowship and the second author has been partially supported by supported by NBHM fellowship.}
\date{August 21, 2020}
\setcounter{tocdepth}{1}

\begin{abstract}
Let $K/F$ be a finite Galois extension of fields with $\textup{Gal}(K/F)=\Gamma$. In an earlier work of Timothy Kohl, the author
enumerated dihedral Hopf-Galois structures acting on dihedral extensions. Dihedral group is one particular 
example of semidirect product of $\mathbb{Z}_n$ and $\mathbb{Z}_2$. In this article we count the number of 
Hopf-Galois structures 
with Galois group $\Gamma$ of type $G$, where $\Gamma,G$ are groups of the form 
$\mathbb{Z}_N\rtimes_\phi\mathbb{Z}_2$ when $N$ is odd with radical of $N$ being a Burnside number. As an 
application we also find the corresponding number of skew braces.
\end{abstract}
\maketitle
\section{Introduction}
\subsection{Hopf-Galois structures} Let $\mathcal{R}$ be a commutative ring with unity and let $\mathcal{H}$
be an $\mathcal{R}$-bialgebra. Then $\mathcal{H}$ will be called an \textit{$\mathcal{R}$-Hopf algebra} if there is an 
$\mathcal{R}$-module homomoprhism $\lambda:\mathcal{H}\rightarrow \mathcal{H}$ (the antipode map), which is 
both an $\mathcal{R}$-algebra and an $\mathcal{R}$-coalgebra antihomomophism such that:
\begin{align*}
    \lambda(h\otimes h')&=\lambda(h)\otimes \lambda(h')\\
    \Delta\lambda(h)&=(\lambda\otimes \lambda)\tau\Delta,
\end{align*}
where $\Delta$ is the comultiplication map and $\tau$ is the switch map $\tau(h_1\otimes h_2)=h_2\otimes h_1$. Now assume that $\mathcal{H}$ is commutative. 
An $\mathcal{R}$-Hopf algebra $\mathcal{H}$ is called a \textit{finite algebra} if it is finitely generated and a projective $\mathcal{R}$-module. Now if $\mathcal{S}$ is an $\mathcal{R}$-algebra which is an $\mathcal{H}$-module, then $\mathcal{S}$ is called an \textit{$\mathcal{H}$-module algebra} if
\begin{align*}
    h(st)=\sum h_{(1)}(s)h_{(2)}(t)\text{ and }h(1)=\epsilon(h)1
\end{align*}
for all $h\in \mathcal{H},s,t\in\mathcal{S}$, where $\Delta(h)=\sum\limits_{(h)}h_{(1)}\otimes h_{(2)}\in\mathcal{H}\otimes\mathcal{H}$ according to Sweedler's (\cite{s}) 
notation and $\epsilon:\mathcal{H}\rightarrow\mathcal{R}$ is the co-unit map.

Then $\mathcal{S}$, a finite commutative $\mathcal{R}$-algebra 
is called an \textit{$\mathcal{H}$-Galois extension} over $\mathcal{R}$ if $\mathcal{S}$ is a left $\mathcal{H}$-module algebra
and the $\mathcal{R}$-module homomorphism
\begin{align*}
    j:\mathcal{S}\otimes_\mathcal{R}\mathcal{H}\rightarrow\textup{End}_\mathcal{R}(\mathcal{S}),
\end{align*}
given by $j(s\otimes h)(s')=sh(s')$ for $s,s'\in\mathcal{S},h\in\mathcal{H}$, is an isomorphism. 
Now we define a Hopf-Galois structure on a separable field extension. Assume $K/F$ is a finite separable field extension. An $F$-Hopf algebra $\mathcal{H}$, with an action on $K$ such that $K$ is an $H$-module algebra
and the action makes $K$ into an $\mathcal{H}$-Galois extension, will be called a \textit{Hopf-Galois structure}
on $K/F$.
\subsection{Greither-Pareigis theory \cite{gp} and Boytt's translation \cite{b}}
Given a group $G$ we define $Holomorph$ of $G$ as a semidirect product $G\rtimes_{\psi} \au(G)$, where $\psi$ is the identity map. Holomorph of a group $G$ (denoted by $\h(G)$) sits inside $\p(G)$ (set of permutations on $G$) as follows
\begin{equation*}
    \h(G) = \{\eta \in \p(G): \eta \text{ normalizes }   \lambda(G)\}
\end{equation*}
where $\lambda$ is the left regular representation. Now we state some results which will help us count the number of Hopf-Galois structures on a given field extension. The following result is due to \cite{gp}.
\begin{lemma}{\cite[Theorem 6.8]{c}}\label{l002}
Let $K/F$ be a separable extension of fields. Let $L$ be its normal closure, $\Gamma=\g(L/F),\Gamma'=\g(L/K)$. Define the set $X=\Gamma/{\Gamma'}$. Then there is a bijection between Hopf-Galois structure on $K/F$ and regular subgroups $G$ of $\p(X)$ normalized by $\lambda(\Gamma)$ where $\lambda$ is the left regular representation.
\end{lemma}
In the proof of the above lemma, given a regular subgroup $G\leq \p (X)$ normalized by $\lambda(\Gamma)$, the Hopf-Galois structure on $K/F$ corresponding to $G$ is $K[G]^{\Gamma}$. 
Here $\Gamma$ acts on $G$ by conjugation inside $\p(X)$ and it acts on $K$ by field automorphism, which induces an action of $\Gamma$ on $K[G]$. This $G$ is called the \textit{type} of the Hopf-Galois extension.
Although Greither-Pareigis theory simplifies the problem of counting the number of
Hopf-Galois structure for a given separable extension, the size of $\p(X)$ is large ($|X|!$) in general.
The next lemma (also known as Boytt's translation) further simplifies the problem by considering regular embedding in $\h(X)$, which is comparatively smaller in size. 
From the proof of {\cite[Proposition 1]{b}} we have the following:

Let $\Gamma$ be a finite group and $\Gamma'\subseteq\Gamma$ be a subgroup. Let $X=\Gamma/\Gamma'$ and $G$ be group of order $|X|$.
Then there is a bijection between the following sets:
\begin{enumerate}
    \item $\{\alpha:G\rightarrow\p(X)\text{ a monomorphism, }\alpha(G)\text{ is regular}\}$
    \item $\{\beta:\Gamma\rightarrow\p(G)\text{ a monomorphism, }\beta(\Gamma')\text{ is stabilizer of the identity of }G\}$
\end{enumerate}

Now assume $K/F$ is a Galois extension with Galois group isomorphic to $\Gamma$. Let $e(\Gamma,G)$ be the 
number of regular subgroups in $\p(\Gamma)$ isomorphic to $G$ which is normalized by $\lambda(\Gamma)$ i.e.
 the number of Hopf-Galois structures on $K/F$ of type $G$. Let $e'(\Gamma,G)$ denote the number of 
subgroups $\Gamma^*$ of $\h(G)$ isomorphic to $\Gamma$, such that the stabilizer in $\Gamma^*$ of $e_G$ is trivial. Then we have the following result.
\begin{lemma}{\cite[See Proposition 1]{b}}\label{l004} With the notations as above we have,
\[e(\Gamma,G) = \dfrac{|\au(\Gamma)|}{|\au(G)|} e'(\Gamma,G).\]
\end{lemma}
Note that $\Gamma^*$ is a regular subgroup of $\h(G)$ implies $\Gamma^*$ has the same cardinality as $G$.
A typical element of $\h(G)$ is of the form  $(g,\zeta)$ where $g\in G,\zeta\in\au(G)$. Hence to say $\Gamma^*$ is a
regular subgroup of $\h(G)$ we have to make sure that $g=e_G$ implies that $\zeta$ is the identity automorphism
of $G$ for every $(g,\zeta)\in\Gamma^*$. We will use this condition to check regular embeddings of the concerned groups in the article.

\subsection{Notations} For  $a,b\in \mathbb{Z}$ we will use $(a,b)$ to denote the g.c.d. of $a,b$. For a 
number $n$, we take $\pi(n)=\{p:p\text{ divides }n,p\text{ prime}\}$. The notation $v_p(n)$ denotes the $p$-valuation of $n$. For $n\in\mathbb{N}$, the radical of $n$ is defined to be product of the distinct primes in $\pi(n)$ which will
be denoted as $\ra(n)$. The symbol $\varphi(n)$ denotes the Euler's totient function at $n\in\mathbb{N}$.
A  number $n\in\mathbb{N}$ is called a Burnside number if $(n,\varphi(n))=1$.

\subsection{Result} 
In \cite{b} the author has proved that If $K/F$ is a finite Galois extension of field of degree $T$, then this extesnion admits a unique Hopf-Galois structure if and only if $T$ is Burnside number. 
Since in our case $N>1$ is odd and hence $2N$ is not Burnside, the extension has at least $2$ Hopf-Galois structure. The number of Hopf-Galois structure for various groups have been studied by 
E.  Campedel et al. \cite{ccc}, T. Kohl \cite{k}, Carnahan S.  et al \cite{cc} et cetera. For an extensive literature review one may look at the PhD thesis of  K. N. Zenouz \cite{z}. In \cite{k}, T. Kohl 
has computed $e(G.G)$ when $G$ is Dihedral group. For $N$ odd we look at groups of order $2N$ of the form $\mi_{k,l}:=\D_{2k}\times \C_l$ where $kl=N,(k,l)=1$, whenever radical of $N$ is a Burnside number. Our main result is the following.
\begin{theorem}\label{t001} Let $K/F$ be a Galois extension of fields with $\g(K/F)\cong\Gamma$ and $N\in\mathbb{N}$ be odd.
If $\Gamma=\mi_{k_1,l_1}$ and $G=\mi_{k_2,l_2}$ where $k_1l_1=k_2l_2=N$ and $\ra(N)$ is a Burnside number,
then the number of Hopf-Galois 
structure on $K/F$ of type $G$ is given by
\begin{align*}
    e(\Gamma,G)=\dfrac{l_1l_2}{(l_1,l_2)\ra(l_1)}\cdot2^{|\pi(k_2)|}.
\end{align*}
\end{theorem}
Rest of the article is organized as follows. In section $2$ we collect background materials needed for the proof of the theorem. Section $3$ is devoted to the proof of the main result. Next in section $4$ we derive some further consequences of the results found in section $3$. All the necessary equations which we use in section $3$ have been included in section $5$.

\section{Preliminaries}
In this section we give complete description of groups of the form $\Z_N\rtimes\Z_2$ and state some basic number theoretic results which will be used in Section $3$ to enumerate regular embeddings.
\subsection{Groups of the form $\Z_{N}\rtimes_\phi\Z_2$, $N$ odd}
Note that if $N=\prod\limits_{t=1}^mp_t^{\alpha_t}$, where $p_i$'s are all distinct primes, then
\begin{align*}
    \mathbb{Z}_N&\cong\bigoplus\limits_{t=1}^m\mathbb{Z}_{p_t^{\alpha_t}},\\
    \text{and }\au(\mathbb{Z}_N)\cong\Z_N^*&\cong\prod_{t=1}^m\mathbb{Z}_{p_t^{\alpha_t}}^*\\
    &\cong\bigoplus\limits_{t=1}^m\mathbb{Z}_{p_t^{\alpha_t-1}(p_t-1)}.
\end{align*}
For $x\in\mathbb{Z}_N$ we have $x=(x_1,x_2,\cdots,x_m)$ where $x_u\in\mathbb{Z}_{p_u^{\alpha_u}}$. We define
$p_u(x)=x_u$ for $p_u\in\pi(N)$.

If $\phi:\Z_2=\{\pm 1\}\rightarrow \au(\Z_{N})$ is a group 
homomorphism with $p_u(\phi(-1))=-1$ for all $p_u\in\pi(N)$, then $\Z_N\rtimes_\phi\Z_2$ is the dihedral group 
of order $2N$ and we will denote this group by $\mathfrak{D}$. When $p_u(-1)=1$ for all
$p_u\in\pi(N)$, then $\Z_n\rtimes_\phi\Z_2$ is the cyclic group 
of order $2N$ and we will denote this group by $\mathfrak{C}$. Now suppose $p_u(-1)=1$ for some $p_u\in\pi(n)$ and $p_{u'}(-1)=-1$ for some $p_{u'}\in\pi(N)$, then the group is isomorphic 
to $\D_{2k}\times \C_l$ for some $k,l\in\mathbb{N}$ with $kl=N$. We denote this group by $\mathfrak{M}_{k,l}$.
We have to consider the regular embeddings for the following cases:
\begin{enumerate}
    \item $\mi_{k_1,l_1}\hookrightarrow\h(\mi_{k_2,l_2})$ where $k_1l_1=k_2l_2=N$ and $(k_1,l_1)=(k_2,l_2)=1$,
    \item $\di\hookrightarrow\h(\mi_{k,l})$ with $k,l>1$,
    \item $\cy\hookrightarrow\h(\mi_{k,l})$ with $k,l>1$,
    \item $\di\hookrightarrow\h(\cy)$
    \item $\cy\hookrightarrow\h(\di)$
    \item $\mi_{k,l}\hookrightarrow\h(\cy)$ with $k,l>1$,
    \item $\mi_{k,l}\hookrightarrow\h(\di)$ with $k,l>1$,
    \item $\di\hookrightarrow\h(\di)$
    \item $\cy\hookrightarrow\h(\cy)$
\end{enumerate}
While counting the regular embeddings we consider the first case and all other cases are special cases of it. We must mention here that
the last two cases have been previously discussed in \cite{k} and \cite{b1} respectively
and our answers match with the results therein.

\subsection{Basic results} 
\begin{lemma}\label{l001}
Let $p>2$ be a prime and $\gamma\equiv 1\mod p$. Define $f_\gamma(0)=0$ and for each $\delta\in\Z_{>0}$ define
\begin{align*}
    f_\gamma(\delta)=\sum\limits_{i=0}^{\delta-1}\gamma^i.
\end{align*}
Then 
\begin{align*}
    f_\gamma(\delta_1)\equiv f_\gamma(\delta_2)\mod p^n\text{ iff }\delta_1\equiv \delta_2\mod p^n.
\end{align*}
\end{lemma}
\begin{proof}
See the proof of Lemma $2.17$ in \cite{ccc}.
\end{proof}
\begin{corollary}\label{c001}
Let $p$ be a prime and $b\in\Z$ such that $b^{p^m}\equiv 1\mod p^n$. Then
\begin{align*}
    p^m|f_b(p^m)\text{and }p^{m+1}\nmid f_b(p^{m}).
\end{align*}
\end{corollary}
\begin{proof}
This follows from the observation that $b\equiv 1\mod p$.
\end{proof}

\section{Regular embeddings and Hopf-Galois structure}\label{re}
We start with a presentation of the group $\mi_{k,l}=\D_{2k}\times \C_{l}$. It is given by
\begin{align*}
    \mi_{k,l}=\left\langle r,s,t:r^k,s^2,t^l,srsr,sts^{-1}t^{-1},rtr^{-1}t^{-1}\right\rangle.
\end{align*}
We have the followings.
\begin{enumerate}
    \item $\h(\mi_{k,l})\cong\h(\C_l)\times \h(\D_{2k})$, since $(k,l)=1$,
    \item $\au(\D_{2k})\cong\left\{\begin{pmatrix}b&a\\0&1\end{pmatrix}:b\in\mathbb{Z}_k^*,a\in\mathbb{Z}_k\right\}$ where
    \begin{align*}
        \begin{pmatrix}b&a\\0&1\end{pmatrix}\cdot r=r^b,\begin{pmatrix}b&a\\0&1\end{pmatrix}\cdot s=r^as,
    \end{align*} 
    \item $\au(\C_l)\cong\mathbb{Z}_l^*$,
    \item \[\h( \mi_{k,l})\cong\left\{\left(\begin{pmatrix}b&a\\0&1\end{pmatrix},r^is^j,\begin{pmatrix}d&c\\0&1\end{pmatrix}\right):\displaystyle{\substack{b\in\Z_l^*,a\in\Z_l,d\in\Z_k^*,c\in\Z_k,\\1\leq i\leq k-1,j=0,1}}\right\},\]
    where $(r^is^j,a)$ corresponds to the element of $\mi_{k,l}$.
\end{enumerate}

Now we want to look at the embeddings $    \Phi:\mi_{k_1,l_1}\rightarrow\h(\mi_{k_2,l_2}).$
We take \begin{align*}
    \mi_{k_1,l_1}&=\left\langle r_1,s_1,t_1:r_1^{k_1},s_1^2,t_1^{l_1},s_1r_1s_1r_1,s_1t_1s_1^{-1}t_1^{-1},r_1t_1r_1^{-1}t_1^{-1}\right\rangle,\\
        \mi_{k_2,l_2}&=\left\langle r_2,s_2,t_2:r_2^{k_2},s_2^2,t_2^{l_2},s_2r_2s_2r_2,s_2t_2s_2^{-1}t_2^{-1},r_2t_2r_2^{-1}t_2^{-1}\right\rangle.
\end{align*}
Let us assume that
\begin{align*}
    \Phi(r_1)&=\left(\begin{pmatrix}b&a\\0&1\end{pmatrix},r_2^is_2^j,\begin{pmatrix}d&c\\0&1\end{pmatrix}\right),\\
    \Phi(s_1)&=\left(\begin{pmatrix}b'&a'\\0&1\end{pmatrix},r_2^{i'}s_2^{j'},\begin{pmatrix}d'&c'\\0&1\end{pmatrix}\right),\\
    \Phi(t_1)&=\left(\begin{pmatrix}b''&a''\\0&1\end{pmatrix},r_2^{i''}s_2^{j''},\begin{pmatrix}d''&c''\\0&1\end{pmatrix}\right).
\end{align*}
We define the set $\mathfrak{V}=\{a,b,i,j,c,d,a',b',i',j',c',d',a'',b'',i'',j'',c'',d''\}$ and refer to the elements of the set as variables.

Note that we can consider the element $a\in\mathbb{Z}_{l_2}$ (resp.
$b\in\mathbb{Z}_{l_2}^*$) to be an element of $\mathbb{Z}_N$ (resp. $\mathbb{Z}_N^*$) by setting $p_u(a)=0$  (resp. $p_u(b)=1$) for all
$p_u\in\pi(N)\setminus\pi(l_2)$. Same treatment will be applicable to all variables in $\mathfrak{V}$ accordingly. 
We observe that $N=(k_1,l_2)(l_1,l_2)(k_1,k_2)(l_1,k_2)$ and 
the four entities in the right are mutually coprime. Thus it is enough to count the total number of 
possibilities of the variables in each of $\mathbb{Z}_{\beta}$, where 
$\beta\in\{(k_1,l_2),(l_1,l_2),(k_1,k_2),(l_1,k_2)\}$. Now we look at the embeddings of the groups inside the holomorph.
We will encounter several equations in this context. We put the equations in section \ref{ap}.

\subsection{Embeddings}
Considering equations \ref{e001}, \ref{e019} (also \ref{e032}) we get that $p_u(b)=1$ for all 
$p_u\in\pi(N)$. Similarly $p_u(d)=1$ for all $p_u\in\pi(N)$. Since $|r_1|=k_1$ we get that $p_u(a)$ is a 
unit whenever $p_u\in\pi((k_1,l_2))$ and $0$ for other primes (this is equivalent to saying 
$|a|=(k_1,l_2)$). Similarly 
\begin{enumerate}
    \item $p_u(i)$ is a unit for $p_u\in\pi((k_1,k_2))$ and $0$ otherwise,
    \item $p_u(c)=0$ whenever $p_u\in\pi(n)\setminus\pi((k_1,k_2))$,
    \item $p_u(i'')$ is a unit for $p_u\in\pi((l_1,k_2))$ and $0$ otherwise,
    \item $p_u(a'')$ is a unit for $p_u\in\pi((l_1,l_2))$ and $0$ otherwise.
\end{enumerate}
Point $(3)$ and $(4)$ follows from Lemma  \ref{l001}.
From equations \ref{e001} and \ref{e019} we have that $p_u(b)=1$ for all $p_u\in\pi(N)$.
In each of these following cases we only determine the coefficients of the variables for the primes relevant to that case.

\textbf{Case I: Inside $\mathbb{Z}_{(k_1,l_2)}$}
Using equations \ref{e015}, \ref{e020} and $b=1$, we have that $a(1+b')=0$ thus
$p_u(b')=-1$ for $p_u\in\pi((k_1,l_2))$. Referring to equation \ref{e011} and $p_u(a)$ is a unit for 
$p_u\in\pi((k_1,l_2))$ we have that $p_u(b'')=1$ for $p_u\in\pi(k_1,l_2)$. Using equation \ref{e037} we have 
$p_u(a'')=0$ for all $p_u\in\pi((k_1,l_2))$ (since $(2,p_u)=1$). All other variables have one possibility since $k_2$ is coprime to $(k_1,l_2)$. Hence
\begin{enumerate}
    \item $a$ has $\varphi(k_1,l_2)$ possibilities,
    \item $a'$ has $(k_1,l_2)$ possibilities.
\end{enumerate}

\textbf{Case II: Inside $\mathbb{Z}_{(l_1,l_2)}$}
Here $p_u(a)=0$ for all $p_u\in\pi((l_1,l_2))$. Using equation \ref{e014} (equiv. \ref{e027}) we have $p_u(b')=\pm 1$ for all $p_u\in\pi((l_1,l_2))$. Considering equation \ref{e024} (equiv. \ref{e037}) and that $1-p_u(b'')$ is a zero divisor for $p_u\in\pi((l_1,l_2))$, using equation \ref{e015} (equiv. \ref{e028}) we get $p_u(b')=1$ which implies that $p_u(a')=0$ in this case.
Since $\ra(N)$ is a Burnside number, from equation \ref{e006} we have that 
\begin{equation}\label{e040}
    p_u(b'')^{p_u^{\alpha_u-1}}=1 \text{ for all }p_u.
\end{equation}
Hence
\begin{enumerate}
    \item $(a',b')$ has $1$ possibility,
    \item $a''$ has $\varphi(l_1,l_2)$ possibilities,
    \item $b''$ has $\dfrac{(l_1,l_2)}{\ra((l_1,l_2))}$ possibilities.
\end{enumerate}
\begin{remark}
Note that the above two cases do not depend on $j'$.
\end{remark}

\textbf{Case III: Inside $\mathbb{Z}_{(k_1,k_2)}$ ($j'=0$)} We have $p_u((i))$ is a unit for all $p_u\in\pi((k_1,k_2))$.  
Combining equations \ref{e018} and \ref{e021} we have $p_u(d')=-1$ for all $p_u\in\pi((k_1,k_2))$.
Since $(k_1,k_2)$ is coprime to $l_1$ and $\ra(N)$ is a Burnside number, using equation \ref{e009}  we conclude that $p_u(d'')=1$ for all $p_u\in\pi((k_1,k_2))$. This implies that $p_u(i'')=p_u(c'')=0$ for all $p_u\in\pi((k_1,k_2))$, since
$(l_1,(k_1,k_2))=1$. Hence
\begin{enumerate}
    \item $i$ has $\varphi((k_1,k_2))$ possibilities
    \item each of $c,i',c'$ has $(k_1,k_2)$ possibilities.
\end{enumerate}

\textbf{Case IV: Inside $\mathbb{Z}_{(l_1,k_2)}$ ($j'=0$)} We have $p_u(i)=p_u(c)=0$ for all $p_u\in\pi((l_1,k_2))$. Note that $p_u(d')=\pm 1$ for all $p_u\in\pi((l_1,k_2))$. Using equation \ref{e025} we have
\begin{align*}
    p_u(i'')(1-p_u(d'))=p_u(i')(1-p_u(d''))\mod p_u^{\alpha_u}\text{ for all }p_u\in\pi((l_1,k_2)).
\end{align*}
Then $p_u(d')=-1$ for some $p_u\in\pi((l_1,k_2))$ implies that
\begin{align*}
    2p_u(i'')=p_u(i')(1-p_u(d''))\mod p_u^{\alpha_u}\text{ for all }p_u\in\pi((l_1,k_2)).
\end{align*}
Since $2p_u(i'')$ is a unit and $1-p_u(d'')$ ($\neq 0$) is a zero divisor, this case does not arise. Hence we get that 
$p_u(d')=1$ for all $p_u\in\pi((l_1,k_2))$. This implies that $p_u(i')=p_u(c')=0$ for all $p_u\in\pi((l_1,k_2))$.
Similar to equation \ref{e040} $p_u(d'')^{p_u^{\alpha_u-1}}=1$ for all $p_u\in\pi((l_1,k_2))$, since $\ra(n)$ is Burnside.
Hence
\begin{enumerate}
    \item $i''$ has $\varphi((l_1,k_2))$ possibilities
    \item $c''$ has $(l_1,k_2)$ possibilities
    \item $d''$ has $\dfrac{(l_1,k_2)}{\ra((l_1,k_2))}$ possibilities.
\end{enumerate}

\textbf{Case V: Inside $\mathbb{Z}_{(k_1,k_2)}$ ($j'=1$)} We have $p_u(d'')=1,p_u(c'')=p_u(i'')=0$ for all $p_u\in\pi((k_1,k_2))$ (as in Case III). Note that $p_u(d')=\pm 1$ for all $p_u\in\pi((k_1,k_2))$. First let us assume that $p_u(d')=1$ for some $p_u$. Then $p_u(c')=0$ and hence $i''\in\mathbb{Z}_{p_u^{\alpha_u}}$. Using equation \ref{e036} we conclude that $p_u(c)=0$.

Next, if $p_u(d')=-1$ for some $p_u$, using equation \ref{e034}
\begin{align*}
    p_u(c')=2p_u(i')\mod p_u^{\alpha_u}.
\end{align*}
Also combining equations \ref{e031}, \ref{e034} we get that $2p_u(i)=-p_u(c)$. Hence
\begin{enumerate}
    \item $i$ has $\varphi((k_1,k_2))$ possibilities,
    \item $(d',i')$ has $2^{|\pi((k_1,k_2))|}(k_1,k_2)$ possibilities.
\end{enumerate}

\textbf{Case VI: Inside $\mathbb{Z}_{(l_1,k_2)}$ ($j'=1$)} We have $p_u(i)=p_u(c)=0$ for all $p_u\in\pi((l_1,k_2))$. Let us assume that $p_u(d')=1$ for some $p_u\in\pi((l_1,k_2))$. Then $p_u(c')=0$. Then using equation \ref{e038} we have
\begin{align*}
    p_u(i')(1-p_u(d''))=2p_u(i'')+p_u(c'')\mod p_u^{\alpha_u}.
\end{align*}
On the other hand if $p_u(d')=-1$ for some $p_u\in\pi((l_1,k_2))$ we get that $p_u(c')=2p_u(i')$ and using equation \ref{e038} we get
\begin{align*}
    p_u(i')(1-p_u(d''))=p_u(c'')\mod p_u^{\alpha_u}.
\end{align*}
Hence in either of the cases $p_u(c'),p_u(c'')$ get fixed by $p_u(d'),p_u(i')$. Hence
\begin{enumerate}
    \item $(d',i')$ has possibilities $2^{|\pi((l_1,k_2))|}(l_1,k_2)$ possibilities,
    \item $i''$ has $\varphi((l_1,k_2))$ possibilities (as in Case IV)
    \item $d''$ has $\dfrac{(l_1,k_2)}{\ra((l_1,k_2))}$ possibilities (as in Case IV).
\end{enumerate}
\subsection{Regularity}
Now we check the regularity of these groups. Note that any element $\sigma$ in the image of $\Phi$ is of the form
\begin{align*}
    \left(\begin{pmatrix}\widetilde{b}&\widetilde{a}\\0&1\end{pmatrix},r_2^{\widetilde{i}}s_2^{\widetilde{j}},\begin{pmatrix}\widetilde{d}&\widetilde{c}\\0&1\end{pmatrix}\right).
\end{align*}
Since $\Phi$ is a homomorphism, this element corresponds to some
$\Phi(r_1)^{\lambda}\Phi(s_1)^{\lambda'}\Phi(t_1)^{\lambda''}$, where $0\leq\lambda\leq l_1-1,0\leq\lambda'\leq1,0\leq\lambda''\leq k_1-1$. First we consider the case when $j'=0$.
Note that in this case
\begin{align*}
    &\Phi(r_1)^\lambda\Phi(s_1)\\
    =&\left(\begin{pmatrix}1&\lambda a\\0&1\end{pmatrix},r_2^{\lambda i},\begin{pmatrix}1&\lambda c\\0&1\end{pmatrix}\right)\left(\begin{pmatrix}b'&a'\\0&1\end{pmatrix},r_2^{i'},\begin{pmatrix}d'&c'\\0&1\end{pmatrix}\right)\\
    =&\left(\begin{pmatrix}b'&\lambda a+a'\\0&1\end{pmatrix},r_2^{\lambda i+i'},\begin{pmatrix}d'&\lambda c+c'\\0&1\end{pmatrix}\right).
\end{align*}
By Chinese Remainder theorem
there exists $0<\lambda<k_1$, such that $\lambda a+a'=0\mod (k_1,l_2)$ and $\lambda i+i'=0\mod(k_1,k_2)$.
Also we have 
\begin{align*}
    p_u(b')&=\begin{cases}
    -1&\text{for all }p_u\in\pi((k_1,l_2))\\
    1&\text{for all }p_u\in\pi((l_1,l_2))
    \end{cases},\\
    p_u(d')&=\begin{cases}
    -1&\text{for all }p_u\in\pi((k_1,k_2))\\
    1&\text{for all }p_u\in\pi((l_1,k_2))
    \end{cases}.
\end{align*}
Since for any such $0<\lambda<k_1$ the element $\Phi(r_1)^\lambda\Phi(s_1)$ is non-trivial. Hence the group generated 
in this case is not regular. Thus $j'=0$ is not possible.

In case $j'=1$, any term of the form $\Phi(r_1)^\lambda\Phi(s_1)\Phi(t_1)^{\lambda''}$ is an
element of a regular subgroup. Hence to check regularity we need to consider the terms $\Phi(r_1)^\lambda\Phi(t_1)^{\lambda''}$ with $\widetilde{a}=0,\widetilde{i}=0$. We have,
\begin{align*}
a''(1+b''+\cdots+(b'')^{\lambda''-1})&=-\lambda a\mod l_2\\
i''(1+d''+\cdots+(d'')^{\lambda''-1})&=-\lambda i\mod k_2.
\end{align*}
Since $p_u(i'')=0$ and $p_u(i)$ is a unit for all $p_u\in\pi((k_1,k_2))$, we get that $p_u(\lambda)=0$ therein. One can also check that $p_u(\lambda)=0$ for all $p_u\in\pi((k_1,l_2))$. Hence $\lambda=0$. Similar as before, using lemma \ref{l001}, we have $\lambda''=0$.
\begin{proposition}\label{p001}
If $\Gamma=\mi_{k_1,l_1}$ and $G=\mi_{k_2,l_2}$, where $k_1l_1=k_2l_2=N$ is an odd number and $\ra(N)$ is a Burnside number 
then
\begin{align*}
    e'(\Gamma,G)=\dfrac{l_1N}{k_1(l_1,l_2)\ra(l_1)}\cdot2^{|\pi(k_2)|}.
\end{align*}
\end{proposition}
\begin{proof}
From the above discussion it is evident that $j'=1$. 
Thus to determine the total number of regular embeddings we have to multiply the number of possibilities obtained in
Cases \textbf{I,II,V,VI} and divide it by $\au(\Gamma)$. Indeed, if $\Phi_1(\Gamma)=\Phi_2(\Gamma)$ for two different 
embeddings $\Phi_1,\Phi_2$ then $\Phi_1^{-1}\Phi_2$ is an automorphism of $\Gamma$. Also if $\xi$ is an automorphism
of $\Phi(\Gamma)$, then $\xi\Phi$ is also a regular embedding of $\Gamma$. Hence
\begin{align*}
    &e'(\mi_{k_1,l_1},\mi_{k_2,l_2})\\
    =&\dfrac{\varphi((k_1,l_2))(k_1,l_2)\varphi((l_1,l_2))(l_1,l_2)\varphi(k_1,k_2)2^{|\pi((k_1,k_2))|}(k_1,k_2)2^{|\pi((l_1,k_2))|}(l_1,k_2)^2\varphi(l_1,k_2)}{\ra((l_1,l_2))|\au(\mi_{k_1,l_1})|\ra((l_1,k_2))}\\
    =&\dfrac{\varphi(N)(l_1,k_2)N2^{|\pi(k_2)|}}{\ra(l_1)|\au(\mi_{k_1,l_1})|}\\
    =&\dfrac{\varphi(N)(l_1,k_2)N2^{|\pi(k_2)|}}{\ra(l_1)\varphi(N)k_1}\\
    =&\dfrac{l_1(l_1,k_2)2^{|\pi(k_2)|}}{\ra(l_1)}\\
    =&\dfrac{l_1N}{k_1(l_1,l_2)\ra(l_1)}\cdot2^{|\pi(k_2)|}.
\end{align*}
Last equality is obvious and this finishes the proof of the proposition.
\end{proof}
\begin{proof}[\textbf{Proof of Theorem \ref{t001}}] Using Lemma \ref{l004}, we have that
\begin{align*}
    &e(\mi_{k_1,l_1},\mi_{k_2,l_2})\\
    =&\dfrac{\au(\mi_{k_1,l_1})}{\mi_{k_2,l_2}}e'(\mi_{k_1,l_1},\mi_{k_2,l_2})\\
    =&\dfrac{l_1l_2}{(l_1,l_2)\ra(l_1)}\cdot2^{|\pi(k_2)|}.
\end{align*}
\end{proof}
\begin{remark}\label{r001}
The case where $l_1=1$ i.e. when $\mi_{k_1,l_1}\cong\D_{2N}$, does not need the assumption that $\ra(N)$ is a Burnside number.
\end{remark}

\section{Further results}
\subsection{Non-classical Dihedral Hopf-Galois structures}
Now consider the group $\D_{2k}\times\C_l$ where $kl=N,(k,l)\neq 1$. We show that $\D_{2N}\not\hookrightarrow\h(\D_{2k}\times\C_l)$. We will need the following lemma
\begin{lemma}\cite[Theorem $3.2$]{bi}\label{l006}
Let $G=H\times K$, where $H$ and $K$ have no common direct factor. Then
\[\au(G)=\left\{\begin{pmatrix}
    A&B\\C&D
\end{pmatrix}\Bigg\vert\begin{tabular}{cc}
     $A\in\au(H)$&$B\in\textup{Hom}(K,Z(H))$  \\
     $C\in\textup{Hom}(H,Z(K))$&$D\in\au(K)$ 
\end{tabular}\right\}.\]
\end{lemma}
\begin{corollary}\label{c002}
The group $\h(\D_{2k}\times\C_l)$ where $kl=N$ is odd ,$(k,l)\neq 1$ does not have any element of order $p^{v_p(N)}$ where $p|(k,l)$.
\end{corollary}
\begin{proof}
Setting $G=\D_{2k},H=\C_l$, we observe that
\begin{enumerate}
    \item $\textup{Hom}(K,Z(H))=1$. Indeed $Z(H)=\C_l$ is a group of odd order, it has no element of order $2$.
    \item $\textup{Hom}(H,Z(K))=1$ since $Z(K)$ is trivial. 
\end{enumerate}
This implies that
\begin{align*}
    \h(\D_{2k}\times\C_l)&=\D_{2k}\times\C_l\rtimes_{id}\au(\D_{2k}\times\C_l)\\
    &\cong \h(\D_{2k})\times\h(\C_l).
\end{align*}
Since none of $\h(\D_{2k}),\h(\C_l)$ has elements of order $p^{v_p(N)}$, the result follows.
\end{proof}
\begin{corollary}\label{c003}
If $N$ is odd then $e(\D_{2N},\D_{2k}\times\C_l)=0$,
whenever $(k,l)\neq 1$.
\end{corollary}
\begin{corollary}\label{c004}
Let $L/K$ be a finite Galois extension with Galois group isomorpic to $\D_{2N}$ where $N$ is odd. Then the
number of Hopf-Galois structures on $L/K$ is at least
\begin{align*}
\displaystyle{\sum\limits_{M=0}^N2^M\chi(N-M)},
\end{align*}
where $\chi(L)$ is the coefficient of $x^L$ in the polynomial $\prod\limits_{p_u\in\pi(N)} (x+p_u^{\alpha_u})$.
\end{corollary}
\subsection{Skew-braces}
\begin{definition}
A left skew brace is a triple $(\Gamma,+,\times)$ where $(\Gamma,+),(\Gamma,\times)$ are groups and satisfy
\[a\times(b+c)=(a\times b)+a^{-1}+(a\times c),\] for all $a,b,c\in\Gamma$.
\end{definition}
Skew braces give non-degenerate set theoretic solutions of the Yang-Baxter equation. It initially appeared in 
the PhD thesis of D. Bachiller and has been studied in \cite{bc}, \cite{cj} et cetera. Skew braces provide group theoretic and ring theoretic methods to understand solutions of the Yang Baxter equations. We will need 
the following result which connects the skew braces and regular subgroups.
\begin{lemma}\cite[Proposition $A.3$]{sv}\label{l007}
Let $\Gamma$ be a group. There exists a bijective correspondence between isomorphism classes of 
skew braces with additive group isomorphic to $\Gamma$ and classes of regular subgroups of $\h(\Gamma)$ 
under conjugation by elements of $\au(\Gamma)$.
\end{lemma}
\begin{corollary}\label{c005}
If $(\Gamma,+)\cong\mi_{k_1,l_1}$ and $(\Gamma,\times)\cong\mi_{k_2,l_2}$, where $k_1l_1=k_2l_2=N$ is an odd number and $\ra(N)$ is a Burnside number 
then the number of skew braces of the form $(\Gamma,+,\times)$ is given by
\begin{align*}
    \dfrac{l_1N}{k_1(l_1,l_2)\ra(l_1)}\cdot2^{|\pi(k_2)|}.
\end{align*}
\begin{proof}
Follows from Proposition \ref{p001}.
\end{proof}
\end{corollary}

\section{Appendix: Equations}\label{ap}
All the notations here are adopted from section \ref{re}. Since $\Phi$ is a homomorphism, we must have the following relations:
\begin{align*}
    \Phi(r_1)^{k_1}&=e_0\\
    \Phi(s_1)^2&=e_0\\
    \Phi(t_1)^{l_1}&=e_0\\
    \Phi(s_1)\Phi(r_1)\Phi(s_1)\Phi(r_1)&=e_0\\
    \Phi(r_1)\Phi(t_1)&=\Phi(t_1)\Phi(r_1)\\
    \Phi(s_1)\Phi(t_1)&=\Phi(t_1)\Phi(s_1),
\end{align*}
where \[e_0=\left(\begin{pmatrix}1&0\\0&1\end{pmatrix},r_2^{0}s_2^{0},\begin{pmatrix}1&0\\0&1\end{pmatrix}\right)\] is the identity element of $\h(\mi_{k_2,l_2})$. 
First we observe that if $j=1$, then $\Phi(r_1)$ has even order. Indeed
\begin{align*}
    &\Phi(r_1)^{2}=\left(\begin{pmatrix}b^{2}&a(1+b)\\0&1\end{pmatrix},r_2^{i(1-d)-c},\begin{pmatrix}d^2&c(1+d)\\0&1\end{pmatrix}\right)\\
    \implies&\Phi(r_1)^{2m+1}=\left(\begin{pmatrix}b^{2m+1}&{a}(1+b+\cdots+b^{2m})\\0&1\end{pmatrix},r_2^{\bar{i}}s,\begin{pmatrix}d^{2m+1}&c(1+d+\cdots+d^{2m})\\0&1\end{pmatrix}\right).
\end{align*}
Since $k_1$ is odd, this possibility does not arise. Similarly we can conclude that $j''=0$, since $l_1$ is odd. Using $\Phi(r_1)^{k_1}=e_0$ we have
\begin{align*}
    &\left(\begin{pmatrix}b^{k_1}&{a}(1+b+\cdots+b^{k_1-1})\\0&1\end{pmatrix},r_2^{i(1+d+\cdots+d^{k_1-1})}s,\begin{pmatrix}d^{k_1}&c(1+d+\cdots+d^{k_1-1})\\0&1\end{pmatrix}\right)\\
    =&\left(\begin{pmatrix}1&0\\0&1\end{pmatrix},r_2^{0}s_2^{0},\begin{pmatrix}1&0\\0&1\end{pmatrix}\right),
\end{align*}
which implies that

\begin{align}
    b^{k_1}&=1\mod{l_2}\label{e001}\\
    {a}(1+b+\cdots+b^{k_1-1})&=0\mod{l_2}\label{e002}\\
    i(1+d+\cdots+d^{k_1-1})&=0\mod{k_2}\label{e003}\\
    d^{k_1}&=1\mod{k_2}\label{e004}\\
    c(1+d+\cdots+d^{k_1-1})&=0\mod{k_2}\label{e005}.
\end{align}
Using $\Phi(t_1)^{l_1}=e_0$ we get that
\begin{align}
    (b'')^{l_1}&=1\mod{l_2}\label{e006}\\
    {a''}(1+b''+\cdots+(b'')^{l_1-1})&=0\mod{l_2}\label{e007}\\
    i(1+d''+\cdots+(d'')^{l_1-1})&=0\mod{k_2}\label{e008}\\
    (d'')^{l_1}&=1\mod{k_2}\label{e009}\\
    c''(1+d''+\cdots+(d'')^{l_1-1})&=0\mod{k_2}\label{e010}.
\end{align}
Using $\Phi(r_1)\Phi(t_1)=\Phi(t_1)\Phi(r_1)$ we get that
\begin{align}
    a(1-b'')&=0\mod l_2\label{e011}\\
    i(1-d'')&=0\mod k_2\label{e012}\\
    c(1-d'')&=0\mod k_2\label{e013}.
\end{align}
 Now we divide the set of equations in two parts considering $j'=0$ and $j'=1$.
 \subsection{Case 1: $j'=0$} 
Using $\Phi(s_1)^2=e_0$ we have
\begin{align}
    (b')^2&=1\mod l_2\label{e014}\\
    a'(1+b')&=0\mod l_2\label{e015}\\
    (d')^2&=1\mod k_2\label{e016}\\
    c'(1+d')&=0\mod k_2\label{e017}\\
    i'(1+d')&=0\mod k_2\label{e018}.
\end{align}
Using $\Phi(s_1)\Phi(r_1)\Phi(s_1)\Phi(r_1)=e_0$ we have
\begin{align}
    b^2&=1\mod l_2\label{e019}\\
    a(b+b')+a'(1+bb')&=0\mod l_2\label{e020}\\
    (i+i')(1+d')&=0\mod k_2\label{e021}\\
    d^2&=1\mod k_2\label{e022}\\
    c(d+d')+c'(1+dd')&=0\mod k_2\label{e023}
\end{align}
Note that $b=1,d=1$ by subsection $4.1$.

Using $\Phi(s_1)\Phi(t_1)=\Phi(t_1)\Phi(s_1)$ we have
\begin{align}
a''(1-b')&=a'(1-b'')\mod{l_2}\label{e024}\\
i''(1-d')&=i'(1-d'')\mod{k_2}\label{e025}\\
c''(1-d')&=c'(1-d'')\mod{k_2}\label{e026}.
\end{align}
 \subsection{Case 2: $j'=1$} 
Using $\Phi(s_1)^2=e_0$ we have
\begin{align}
    (b')^2&=1\mod l_2\label{e027}\\
    a'(1+b')&=0\mod l_2\label{e028}\\
    (d')^2&=1\mod k_2\label{e029}\\
    c'(1+d')&=0\mod k_2\label{e030}\\
    i'(1-d')&=c'\mod k_2\label{e031}.
\end{align}
Using $\Phi(s_1)\Phi(r_1)\Phi(s_1)\Phi(r_1)=e_0$ we have
\begin{align}
    b^2&=1\mod l_2\label{e032}\\
    a(b+b')+a'(1+bb')&=0\mod l_2\label{e033}\\
    (i+i')(1-d')&=d'c+c'\mod k_2\label{e034}\\
    d^2&=1\mod l_2\label{e035}\\
    c(d+d')+c'(1+dd')&=0\mod l_2\label{e036}
\end{align}
Note that $b=1,d=1$ by subsection $4.1$.

Using $\Phi(s_1)\Phi(t_1)=\Phi(t_1)\Phi(s_1)$ we have
\begin{align}
a''(1-b')&=a'(1-b'')\mod{l_2}\label{e037}\\
i'-i''d'&=i''+i'd''+c''\mod{k_2}\label{e038}\\
c''(1-d')&=c'(1-d'')\mod{k_2}\label{e039}.
\end{align}

\end{document}